\documentclass{amsart}

\usepackage[english]{babel}
\usepackage[utf8x]{inputenc}
\usepackage[T1]{fontenc}
\usepackage{appendix,amssymb,amsthm,mathtools,mathdots,nicefrac,tikz,enumitem} 
\usetikzlibrary{math,graphs,cd}
\usepackage{xcolor}

\usepackage[a4paper,top=3cm,bottom=2cm,left=3cm,right=3cm,marginparwidth=1.75cm]{geometry}

\usepackage{graphicx}
\graphicspath{{figures/}}
\usepackage{comment} 
\usepackage[colorlinks=true, allcolors=blue]{hyperref}
\usepackage{algorithm2e}
\usepackage{longtable}

\DeclarePairedDelimiter{\floor}{\lfloor}{\rfloor}
\DeclarePairedDelimiter{\abs}{|}{|}

\usepackage[colorinlistoftodos]{todonotes}

\usepackage{tikz-cd}
\usepackage{mathrsfs}
\usepackage{todonotes}
\usepackage{hyperref}

\newtheorem{theorem}{Theorem}[section]
\newtheorem{corollary}[theorem]{Corollary}
\newtheorem{lemma}[theorem]{Lemma}
\newtheorem{computeraided}[theorem]{Computer Aided Theorem}
\newtheorem{proposition}[theorem]{Proposition}

\newcommand{\Z}{\mathbb{Z}}

\DeclareMathOperator{\Div}{Div}

\DeclareMathOperator{\outdeg}{outdeg}

\DeclareMathOperator{\val}{val}

\DeclareMathOperator{\gon}{gon}

\DeclareMathOperator{\scw}{scw}

\DeclareMathOperator{\ci}{Ci}
\DeclareMathOperator{\supp}{supp}
\DeclareMathOperator{\sn}{sn}

\title{The gonality of circulant graphs}
\author[Cenek]{Lisa Cenek}
\author[Ferguson]{Lizzie Ferguson}
\author[Gebre]{Eyobel Gebre}
\author[Marcussen]{Cassandra Marcussen}
\author[Meintjes]{Jason Meintjes}
\author[Morrison]{Ralph Morrison}
\author[Ostermeyer]{Liz Ostermeyer}
\author[Ramakrishna]{Shefali Ramakrishna}
\date{}

\begin{document}

\begin{abstract}
    The gonality of a graph measures how difficult it is to move chips around the entirety of a graph according to certain chip-firing rules without introducing debt.  In this paper we study the gonality of circulant graphs, a class of vertex-transitive graphs that can be specified by their number of vertices together with a list of cyclic adjacency relations satisfied by all vertices.  We provide a universal upper bound on the gonality of all circulant graphs with a fixed adjacency list, which holds irrespective of the number of vertices.  We use this upper bound together with computational methods to determine that the gonality of the \(4\)-regular Harary graph on \(n\) vertices is \(10\) for \(n\geq 16\).  As a special case, this gives the gonality of sufficiently large antiprism graphs to be \(10\).
\end{abstract}

\maketitle

\section{Introduction}
In their seminal paper \cite{BAKER2007766}, Baker and Norine developed a theory of chip-firing games on graphs in parallel to divisor theory on algebraic curves.  This allowed for graph-theoretic analogs of algebro-geometric notions to be defined, such as the \emph{gonality} of a graph, first introduced in \cite{baker}. Roughly speaking the gonality of a graph is the minimum number of chips one can place on a graph so that no matter where a unit of debt is placed on the graph, one can eliminate all debt through chip-firing moves.

In this paper we study the gonality of \emph{circulant graphs}. Given an integer $n$ and an \emph{adjacency list} $J=\{j_1,j_2,\ldots,j_k\} \subset 
\Z^{+}$, the \emph{circulant graph} $\ci_n(J)$ on vertices $v_1, v_2, \ldots, v_{n}$ has vertex $v_i$ adjacent to vertex $v_m$ if and only if $\abs{i-m}\mod{n} = j_\alpha$ for some $j_\alpha\in J$. Circulant graphs form a broad family encompassing many familiar examples: cycle graphs, complete graphs, and Harary graphs, to name a few.  Our main result is the following upper bound on the gonality of circulant graphs, which relies only on \(J\) and not on \(n\) (for sufficiently large \(n\)).

\begin{theorem}\label{theorem:universal_upper_bound}
Let $J = \{j_1,j_2,\ldots,j_k\} \subset \mathbb{Z}^+$ with $1 \leq j_1 < j_2 < \ldots < j_k$, and let $n$ be a positive integer such that \(\ci_n(J)\) is connected.  We have
\[
\gon(\ci_n(J)) \leq 2 \sum_{\alpha} j_\alpha^2.
\]
\end{theorem}

A  corollary is that for any fixed \(J\), the set \(\{\gon(\ci_n(J))\,|\,n\in\mathbb{Z}\}\) is bounded.  However, as gonality can behave unpredictably, it does not immediately follow that the gonality of \(\ci_n(J)\) eventually stabilizes as \(n\) increases.  For the \(2k\)-regular \emph{Harary graphs} \(H_{2k,n}:=\ci_n(\{1,2,\ldots,k\})\), we can prove stabilization occurs.

\begin{theorem}
     For every \(k\in\mathbb{Z}^+\), there exists \(n_k\) such that for all \(n\geq n_k\), we have
     \[\gon(H_{2k,n})=\gon(H_{2k,n_k}).\]
\end{theorem}

Note that this result is not effective, in that we do not know the value of \(n_k\) or of \(\gon(H_{2k},n_k)\).  For the case of \(k=2\), we make our result effective.

\begin{theorem} We have
\begin{eqnarray}
    \begin{aligned}
        \gon(H_{4,n}) = & 
        \begin{cases}
        \displaystyle \floor*{\frac{n}{4}} + \floor*{\frac{n+1}{4}} + 2,  & n < 16 \\
        10,                                                 & n \geq 16
        \end{cases}
    \end{aligned}
\end{eqnarray} 
\end{theorem}

We structure our paper as follows. In Section \ref{section:background} we cover the necessary background material, useful lemmas, and formally introduce divisor theory as it pertains to gonality. In Section \ref{section:circulantbound} we prove the main theorem and resulting corollary providing the upper bound. In Section \ref{section:harary} we apply our results to Harary graphs, and explore the limitations of other methods in Appendix \ref{appendix}.

\medskip
\noindent \textbf{Acknowledgements.} The authors thank Williams College and the SMALL REU for their support, as well as the NSF for support via grants DMS-1659037 and DMS-2011743.  The authors also thank Thor Gabrielsen and S\o ren Newman-Taylor for helpful discussions on their work on line-like graphs, which provides an independent method of proof for a result similar to our Theorem \ref{theorem:universal_upper_bound}; and Noam Pasman, for helpful comments on an earlier draft.

\section{Background and definitions}\label{section:background}

 Throughout this paper, we let $G=(V,E)$ be an undirected, connected, loopless multigraph with finite vertex set $V(G)$ and finite multiset $E(G)$ of edges. The \emph{valence} of a vertex $v$, denoted by $\val(v)$, is the number of edges incident to it\footnote{Valence is also known as \emph{degree} in graph theory; however, degree has a different meaning in the context of chip-firing games.}. For subsets of vertices $A,B \subset V(G)$, we denote by $E(A,B)$ the multiset of edges with one end in $A$ and the other in $B$. Thus, $\val(v) = \abs{E(\{v\}, V(G)\setminus\{v\})}$. When context makes it clear we will refer to the edges between vertices by $E(u,v) = E(\{u\},\{v\})$ for simplicity. For a subset $S \subset V(G)$ and $v \in S$, the \emph{outdegree of $v$}, $\outdeg_S(v)$, is the number of edges leaving $S$ from $v$, i.e. $\abs{E(v, S^C)}.$ Further, we denote by $G[S]$ the subgraph \emph{induced} by  vertices $S \subset V(G)$ with edge set $E(S,S)$. If, for $S \subset V(G)$, no two vertices are adjacent we call $S$ an \emph{independent set}. The cardinality of the largest such set is the \emph{independence number $\alpha(G)$} of $G$.

The primary graphs of focus in this paper are the \emph{circulant graphs}.  Let $n$ be a positive integer and $J=\{j_1,\ldots,j_k\}\subset \{1,2,\ldots,\lfloor n/2\rfloor\}$ be a set of integers with $j_1<j_2<\cdots<j_k$.  The circulant graph $\textrm{Ci}_n(J)$ has vertex set $\{v_1,\ldots,v_n\}$ with $v_i$ adjacent to $v_m$ if and only if $|i-m|\mod n=j_\alpha$ for some $j_\alpha\in J$.  Many familiar graphs have such a structure; for instance, the cycle graph $C_n$ is $\textrm{Ci}_n(\{1\})$, and the complete graph $K_n$ is $\textrm{Ci}_n(\{1,2,\ldots,\lfloor n/2\rfloor\})$.

Another family of circulant graphs are the \emph{Harary graphs}. These were first introduced in \cite{harary} to study the maximum connectivity of a graph with a given number of vertices and edges.  For a positive integer $n\geq 3$ and an even integer $2\leq k\leq n-1$, the Harary graph $H_{k,n}$ is the circulant graph $C_n(\{1,2,\ldots,k/2\})$; that is, $H_{k,n}$ consists of $n$ vertices arranged in a circle, with each vertex adjacent to the nearest $k$ vertices.  For $n$ even and $k$ odd, we can also define the Harary graph $H_{k,n}$ to be $C_n(\{1,2,\ldots,\lfloor k/2\rfloor, n/2\})$; in other words, it is the Harary graph $H_{k-1,n}$ with each vertex also adjacent to the opposite vertex in the circle.  Note that $H_{k,n}$ is a $k$-regular graph, meaning Harary graphs give a proof by construction that, as long as $k\leq n-1$ and as long as $k$ and $n$ are not both odd, there exists a $k$-regular graph on $n$ vertices.

The $4$-regular Harary graph $H_{4,11}$ is illustrated on the left in Figure \ref{figure:harary_4_11}, with the Harary graph $H_{4,12}$ illustrated twice on the right.  The rightmost depiction of $H_{4,12}$ has the structure of an antiprism, consisting of two cycles rotated to be out of sync with each vertex on one cycle connected to the nearest two on the other.  Thus the graphs $H_{4,n}$ for $n$ even are sometimes referred to as the antiprism graphs. 

\begin{figure}[hbt]
    \centering
    \includegraphics[scale=0.8]{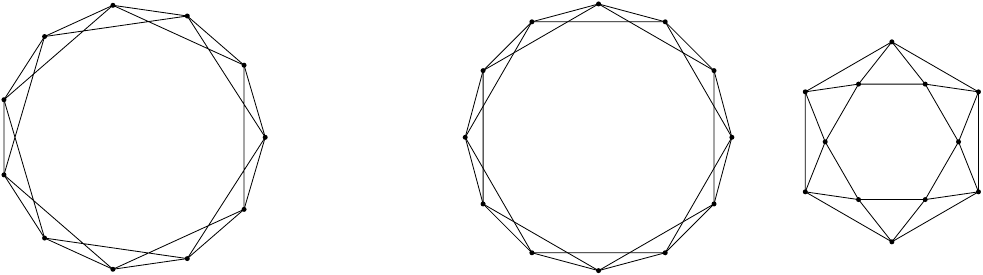}
    \caption{The Harary graph $H_{4,11}$ on the left, and two depictions of the Harary graph $H_{4,12}$ on the right.  The rightmost drawing is the antiprism drawing of $H_{4,12}$.}
    \label{figure:harary_4_11}
\end{figure}

We now present the necessary definitions and results from divisor theory on graphs, mostly following the notation and conventions of \cite{sandpiles}. A \emph{divisor} on graph $G$ is a finite formal sum of the vertices of $G$: \[D=\sum_{v \in V(G)} D(v) \cdot v.\] 
Equivalently, a divisor is an assignment of integers $D(v)$ to the vertices of $G$.  We intuitively describe a divisor as a chip-placement, which places $D(v)$ chips on the vertex $v$ (with negative chips representing debt). The collection of all divisors on a graph $G$ forms the free abelian group $\Div(G)$ under vertex-wise addition, \(\Div(G) \simeq \Z^{\abs{V(G)}}.\)

A divisor $D$ is called \emph{effective} if the number of chips on each vertex is non-negative; that is, if $D(v) \geq 0$ for all vertices $v \in V(G)$. The set of all effective divisors is denoted $\Div^{+}(G)$. The \emph{degree} of a divisor $D$ is the total number of chips on graph $G$: \[\deg(D) = \sum_{v \in V(G)} D(v).\] We denote the set of all divisors of degree $d$ as $\Div_d(G)$ The \emph{support} of an effective divisor $D$ is the set of vertices with nonzero coefficients: \(\supp(D) = \{v \in V(G) \mid D(v) > 0 \}.\)

A \emph{chip-firing move at $v$} transforms a divisor $D$ into a divisor $D'$ by sending chips to the neighbors of $v$, one chip along each edge incident to $v$. More formally, given divisor $D$ and chip-firing vertex $v$, we can write the resulting divisor as \[ D' = \left(D(v) - \val(v)\right)\cdot v + \sum_{u \neq v \in V(G)} \left(D(u) + \abs*{E(u,v)}\right)\cdot u.\] 
If two divisors differ by a sequence of chip-firing moves, we say they are \emph{equivalent}.  Note that the degree of a divisor is preserved under equivalence.

Since the order in which vertices are fired does not affect the final divisor, we may \emph{subset-fire} a subset $S \subseteq V(G)$.  For an example of a subset-firing move, consider the divisor illustrated on the left on $H_{4,11}$ in Figure \ref{figure:harary_subset_fire_example}.  By firing the circled subset of vertices, we obtain the equivalent divisor to the right.  Note that each vertex fired loses precisely as many chips as it has edges leaving the set $S$; and each vertex not fired gains precisely as many chips as it has edges entering the set $S$.  We call a subset-firing move \emph{legal} if it does not  introduce debt on any vertex.  It turns out that if $D$ and $D'$ are equivalent effective divisors, there exists a sequence of legal subset-firing moves to transform $D$ into $D'$ \cite[Corollary 3.11]{db}.  Thus even though individual chip-firing moves might introduce intermediate debt, we can choose subset-firing moves to avoid this as long as our two divisors are both effective.

\begin{figure}[hbt]
    \centering
    \includegraphics{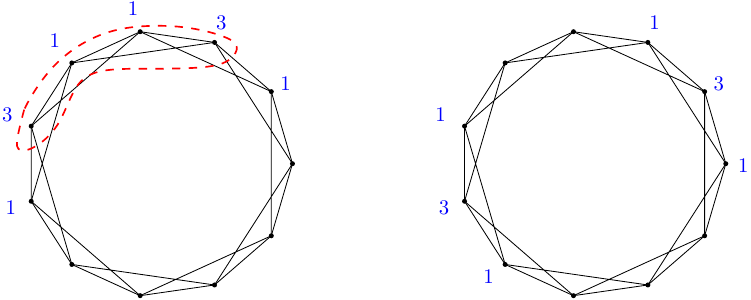}
    \caption{Two equivalent divisors on $H_{4,11}$.  The right is obtained from the left by subset-firing the circled vertices.}
    \label{figure:harary_subset_fire_example}
\end{figure}


We now define the \emph{rank} $r(D)$ of a divisor $D$.  If $D$ is not equivalent to any effective divisor (that is, if we cannot perform chip-firing moves to eliminate the debt present in $D$), we set $r(D)=-1$.  Otherwise, the rank $r(D)$ is the maximum integer $r$ such that, for any effective divisor $E$ of degree $D$, the divisor $D-E$ is equivalent to some effective divisor.  Intuitively, the rank measures how much added debt the divisor $D$ can eliminate, regardless of where that debt is placed.

The \emph{(divisorial) gonality} $\gon(G)$ of a graph $G$ is the minimum degree of a positive rank divisor on $D$.  Equivalently, $\gon(G)$ is the minimum degree of a divisor that can place a chip on any vertex without any other vertex being in debt.

We claim the divisor $D$ in Figure \ref{figure:harary_subset_fire_example} has positive rank.  To see this, note that by firing larger and larger subsets (stretching from one $3$-chip vertex to another), we can effectively translate two copies of the $1-3-1$ arrangement of chips around the graph without introducing debt anywhere, thus letting us place a chip on any vertex.  It follows that $\gon(H_{4,11})\leq \deg(D)=10$.

It turns out this chip-placement is far from optimal on $H_{4,11}$.  A better strategy is to choose an independent set of vertices $S\subset V(G)$ with $|S|=\alpha(G)$ (recall that $\alpha(G)$ is the independence number of $G$), and place a single chip on each vertex in $S^C$.  This yields a positive rank divisor, by the following result.

\begin{proposition}[Theorem 3.1 in \cite{gonality_of_random_graphs}] If $G$ is a simple graph with $n$ vertices, then $\gon(G)\leq n-\alpha(G)$, with a positive rank divisor obtained by placing a chip on each vertex in the complement of a maximum independent set.
\end{proposition}

Combined with the following lemma which shows that $\alpha(H_{4,11})=3$, this implies that $\gon(H_{4,11})\leq 8$, an improvement over the upper bound of $10$.

\begin{lemma} For $k=2\ell$ even, the independence number of the Harary graph $H_{k,n}$ is
\[\alpha(H_{k,n})=\left\lfloor \frac{n}{\ell+1}\right\rfloor\]
\end{lemma}

\begin{proof} Note that the set of vertices $S=\{v_1,v_{1+(\ell+1)},v_{1+2(\ell+1)},\ldots,v_{1+s(\ell+1)}\}$ is an independent set, as long as $1+s(\ell+1)\leq n-\ell$.  Equivalently this set is independent for $s+1\leq \frac{n}{\ell+1}$; since $s$ is an integer this is equivalent to $s+1\leq \left\lfloor \frac{n}{\ell+1}\right\rfloor$.  Since $|S|=s+1$, we have $\alpha(H_{k,n})\geq \left\lfloor \frac{n}{\ell+1}\right\rfloor$ by choosing $s$ as large as possible.

For the upper bound, let $S$ be any maximum independent set.  Since $H_{k,n}$ is vertex-transitive, we may assume without loss of generality that $v_1\in S$. This means that $S-\{v_1\}$ can only contain vertices in $\{v_{1+(\ell+1)},\ldots,v_{n-\ell}\}$, since all other vertices are neighbors of $v_1$.  Partition these remaining vertices into intervals of $\ell+1$ integers:  
\[\{v_{1+(\ell+1)},\ldots,v_{2(\ell+1)}\},\{v_{1+2(\ell+1)},\ldots,v_{3(\ell+1)}\},\ldots, \{v_{1+s(\ell+1)},\ldots,v_{n-\ell}\}\]
where $s$ is chosen maximally so that $1+s(\ell+1)\leq n-\ell$, i.e. so that $s+1=\left\lfloor \frac{n}{\ell+1}\right\rfloor$.  There are thus $\left\lfloor \frac{n}{\ell+1}\right\rfloor-1$ of these intervals, and since all vertices in the same interval are adjacent to one another, each interval has at most one element from $S$.  Combined with $v_1$, this gives us that $\alpha(G)=|S|\leq\left\lfloor \frac{n}{\ell+1}\right\rfloor$, completing the proof.
\end{proof}

Although the degree $10$ divisor from Figure \ref{figure:harary_subset_fire_example} was not optimal for $H_{4,11}$, the remarkable thing about it is that a similar chip placement will have positive rank on \emph{any} Harary graph $H_{4,n}$, since the subset-firing strategy can be repeated as many times as necessary.  It follows that $\gon(H_{4,n})\leq 10$ for all $n$.  Indeed, this is the content of Theorem \ref{theorem:universal_upper_bound} in the special case of $4$-regular Harary graphs.


\section{Bounding the gonality of circulant graphs}\label{section:circulantbound}

In this section we present a divisor that has positive rank on any circulant graph with adjacency list $J$, whose degree does not depend on the number of vertices $n$.

\begin{theorem}\label{thm:automaton}


Given $J = \{j_1,j_2,\ldots,j_k\} \subset \mathbb{Z}^+$ with $1 \leq j_1 < j_2 < \ldots < j_k$ and an integer $n$, consider the circulant graph $\ci_n(J)$. The divisor $D$ defined by
\begin{eqnarray}
    \begin{aligned}
        D(v_\alpha) = & 
        \begin{cases}
        \displaystyle \sum_{i=1}^{k} \max \left\{0,j_i - \abs*{j_k - \alpha} \right\}  & \alpha \in [1,2j_k - 1] \\[1.25em]
        D(v_{n-\alpha+1}) & \alpha \in [n-2j_k+1, n]      \\
        0 & \text{otherwise}
        \end{cases}
    \end{aligned}
\end{eqnarray} 
has positive rank.
\end{theorem}

Note that when $G$ is the $4$-regular Harary graph $H_{4,n}$, this divisor $D$ is precisely the one illustrated on the left in Figure \ref{figure:harary_subset_fire_example}.

\begin{proof}
We will assume $n \geq 4j_k - 3$; otherwise our divisor places a chip on every vertex, and so certainly has positive rank.

    First note that $D(v_\alpha)>0$ for $\alpha\in [1,2j_{k}-1]$, since for all such $\alpha$ we have that $j_k-|j_k-1-\alpha|>0$.  Symmetrically $D(v_\alpha)>0$ for $\alpha \in [n-2j_k+1, n]$.  Thus it remains to show that we can place a chip on any other vertex without debt elsewhere.

    Consider the set $S=\{v_{\alpha}\,|\, 1\leq\alpha\leq j_k \textrm{ or }n-j_k+1\leq \alpha\leq n\}$.  Transform $D$ into a divisor $D'$ by set-firing $S$.  We analyze in several cases the values of $D'(v_\alpha)$. 
    \begin{itemize}
        \item Note that $D(v_1)=1$, and that firing $S$ loses $D(v_1)$ exactly one chip (to $v_{j_k+1}$), so $D'(v_1)=0$.  Similarly, Firing $S$ gains $D(v_{2j_k})$ one chip (from $v_{j_k}$), so $D'(v_{2j_k})=1$.  Note that $D(v_{2j_k-1})=1$, so $D'(v_{2j_k})=D(v_{2j_k-1})$.
        \item  Let $2\leq \alpha\leq j_k$.  The vertex $v_\alpha$ loses $1$ chip for each $i$ such that $\alpha+j_i\geq j_{k}+1$, since those values of $i$ correspond precisely to the edges incident to $v_\alpha$ that are incident to a vertex outside of $S$.  Equivalently, these are the values of $i$ such that $j_i-(j_k-\alpha)\geq 1$. For this range of $\alpha$, this is equivalent to $j_i-|j_k-\alpha|\geq 1$.  In other words, $v_\alpha$ loses $1$ chip for each term that contributes a nonzero amount to
        \[\sum_{i=1}^{k} \max \left\{0,j_i - \abs*{j_k - \alpha} \right\},\]
        meaning that 
        \[D'(v_\alpha)=\sum_{i=1}^{k} \max \left\{0,j_i - \abs*{j_k - \alpha}-1 \right\}=\sum_{i=1}^{k} \max \left\{0,j_i - \abs*{j_k - (\alpha-1)} \right\}=D(v_{\alpha-1}).\]

        \item Let $j_{k}+1\leq \alpha\leq 2j_k-1$.  The vertex $v_\alpha$ gains $1$ chip for each $i$ such that $\alpha-j_i\leq j_k$, since those values of $i$ correspond precisely to the edges incident to $v_\alpha$ that are incident to a vertex within $S$.  Equivalently, these are the values of $i$ such that $j_i-(\alpha-j_k)\geq 1$.  For this range of $\alpha$, this is equivalent to $j_i-|j_k-\alpha|\geq 1$.  In other words, $v_\alpha$ gains $1$ chip for each term that contributes a nonzero amount to
        \[\sum_{i=1}^{k} \max \left\{0,j_i - \abs*{j_k - \alpha} \right\},\]
        meaning that 
        \[D'(v_\alpha)=\sum_{i=1}^{k} \max \left\{0,j_i - \abs*{j_k - \alpha}+1 \right\}=\sum_{i=1}^{k} \max \left\{0,j_i - \abs*{j_k - (\alpha-1)} \right\}=D(v_{\alpha-1}).\]
        \item Symmetric arguments to the above show that $D'(v_n)=0$, and that $D'(v_\alpha)=D(v_{\alpha+1})$ for $\alpha\in [n-2j,n-1]$.
        \item Finally, for all other values of $\alpha$, we have $D'(v_\alpha)=0$, since such vertices had no chips to start, and are adjacent to no vertices within the set $S$.
    \end{itemize}
    In summary, the effect of the set-firing move $S$ is to translate the chips placed on $v_1$ through $v_{2j_k-1}$ one unit clockwise, and the chips placed on $v_n$ through $v_{n-2j_k+1}$ one unit counterclockwise.
    
    This process can be repeated by firing slightly larger sets (adding one vertex clockwise and one unit counterclockwise each time), allowing us to place a chip on any vertex of the graph without intermediate debt being introduced elsewhere.  Thus $D$ has positive rank.
\end{proof}

We are now ready to prove Theorem \ref{theorem:universal_upper_bound}, which gives a universal upper bound of $2 \sum_{\alpha} j_\alpha^2$ for the gonality of $\ci_n(J)$ with $J = \{j_1,j_2,\ldots,j_k\} \subset \mathbb{Z}^+$ with $1 \leq j_1 < j_2 < \ldots < j_k$.


\begin{proof}[Proof of Theorem \ref{theorem:universal_upper_bound}]
Letting $D$ be the positive rank divisor from Theorem \ref{thm:automaton}, we have $\gon(\ci_n(J))\leq \deg(D)$.  For the remainder of the proof we assume $n\geq 4j_k - 3$, as this ensures all vertices appearing in our description of $D$ actually occur; if $n<4j_k-3$, then $\deg(D)$ will be smaller than the formulas below.

By construction, the degree of our divisor is 
\begin{equation}
\begin{split}
    2 \left( 2 \left( \sum_{i=1}^{j_1} i + \sum_{i=1}^{j_2} i + \ldots + \sum_{i=1}^{j_k} i \right) - \sum_{i=1}^{k} j_i  \right)
   & = 2 \left( 2 \left( \sum_{\alpha=1}^{k} \sum_{i=1}^{j_\alpha} i \right) - \sum_{i=1}^{k} j_i  \right)\\
& = 2 \left( 2 \left( \sum_{\alpha=1}^{k} \frac{j_\alpha(j_\alpha + 1)}{2} \right) - \sum_{i=1}^{k} j_i  \right)\\
 & = 2 \left( \sum_{\alpha=1}^{k} j_\alpha(j_\alpha + 1) - \sum_{i=1}^{k} j_i \right)\\
 & = 2 \sum_{\alpha}^k j_\alpha^2. \qedhere
\end{split}
\end{equation}
\end{proof}



\begin{corollary}\label{corollary:harary_even} For $k$ even, the Harary graph $H_{k,n}$ has gonality at most $k(k+1)(k+2)/12$.
\end{corollary}
\begin{proof}
    Since $H_{k,n}=\ci_n(\{1,\ldots,k/2\})$, we have
\begin{align*}
    \gon(H_{k,n})\,\leq&\, 2\sum_{i=1}^{k/2} i^2\,=\, 2\cdot \frac{(k/2)((k/2)+1)(2(k/2)+1)}{6}\,=\,\frac{k(k+1)(k+2)}{12}. \qedhere
\end{align*}
\end{proof}

Note that plugging in $k=4$ does indeed yield $10$, as expected from our divisor in Figure \ref{figure:harary_subset_fire_example}.  In the next section, we will see how to prove that this value of $10$ is optimal once $H_{4,n}$ has enough vertices.

We can also consider the case of odd Harary graphs $H_{k,n}$ with $k$ odd and $n$ even, which can be written as $H_{k,n}=\ci_n(\{1,\ldots,(k-1)/2,n/2\})$.  For fixed $k$ and varying $n$, this collection of graphs does not fall under the purview of Theorem \ref{theorem:universal_upper_bound}, since the adjacency list depends on the value of $n/2$.  Nonetheless we can find a similar upper bound on gonality for these and similar circular graphs.

\begin{theorem}\label{theorem:n_over_2} Consider for $n$ even the circulant graph $\ci_n(J(n))$, where $J(n)=\{j_1,j_2,\ldots,j_k,n/2\}$.  The gonality of this graph is at most
\[4 \sum_{\alpha} j_\alpha^2.\]
\end{theorem}

\begin{proof} Consider first the graph $\ci_n(J(n)-\{n/2\})$.  Construct the divisor from Theorem \ref{thm:automaton}, which is centered on the vertices $v_1$ and $v_n$.  On $\ci_n(J(n))$, build this same divisor, and then add another copy, this one centered on $v_{(n/2)+1}$ and  $v_{(n/2)}$.  This divisor has degree $4 \sum_{\alpha} j_\alpha^2$.  To see it has positive rank, we follow the same firing script as from Theorem \ref{thm:automaton}, except every time we fired a strip of vertices centered on $v_1$ and $v_n$, we now fire both that and the mirror image of the strip, centered on $v_{(n/2)}+1$ and  $v_{(n/2)-1}$.  Since a vertex $v_i$ is fired precisely when $v_{i+n/2}$ is fired, no chips are transferred from a vertex to its antipodal neighbor.  Thus the divisor spreads chips throughout the graph without introducing debt, meaning it has positive rank.
\end{proof}

This theorem, together with the proof of Corollary \ref{corollary:harary_even}, gives the following.
\begin{corollary}\label{corollary:harary_odd}
    For $k$ odd and $n$ even, the Harary graph $H_{k,n}$ has gonality at most $(k-1)k(k+1)/6$.
\end{corollary}

As a concrete example, consider the $3$-regular Harary graph $H_{5,n}$, illustrated in Figure \ref{figure:5-regular_harary} for $n=20$.  The divisor on the left is the one prescribed in the proof of Theorem \ref{theorem:n_over_2}; the divisor on the right is obtained by firing the circled set of vertices.

\begin{figure}
    \centering
    \includegraphics[scale=0.6]{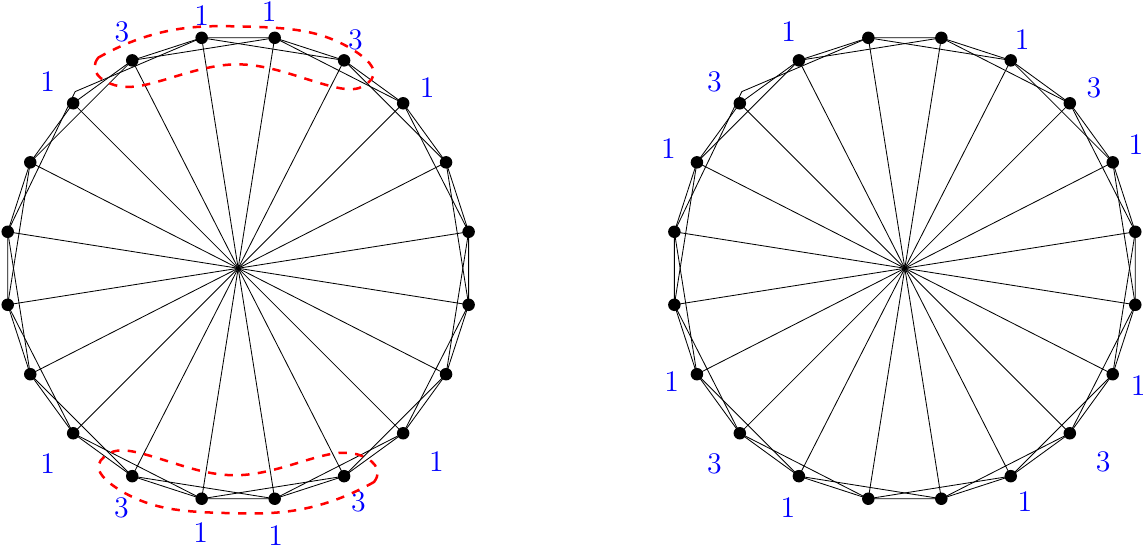}
    \caption{A positive rank divisor on the Harary graph $H_{5,20}$}
    \label{figure:5-regular_harary}
\end{figure}

We close this section with the gonality of $3$-regular Harary graphs.

\begin{theorem}For $n$ even and $n\geq 4$, we have
\begin{eqnarray*}
    \begin{aligned}
        \gon(H_{3,n}) = & 
        \begin{cases}
        3,  & n \leq 6 \\
        4,                                                 & n \geq 8
        \end{cases}
    \end{aligned}
\end{eqnarray*} 
    
\end{theorem}

\begin{proof}
    Note that $H_{3,4}$ is the complete graph $K_4$, and $H_{3,6}$ is the complete bipartite graph $K_{3,3}$.  These both have gonality $3$ \cite[Example 4.3]{debruyn2014treewidth}.

    Now assume $n\geq 8$.  We know by Corollary \ref{corollary:harary_odd} that $\gon(H_{3,n})\leq 2\cdot 3\cdot 4/6=4$.  For a lower bound, consider $H_{3,8}$.  This is the Wagner graph, and has treewidth equal to $4$; indeed, it is one of the four minimal forbidden minors for graphs of treewidth at most $3$ \cite{forbidden_minors_3trees}.  Since $H_{3,8}$ is a minor of $H_{3,n}$ for every $n\geq 8$, we know that $H_{3,n}$ has treewidth at least $4$.  Treewidth is a lower bound on graph gonality \cite[Theorem 1.1]{debruyn2014treewidth}, so $\gon(H_{3,n})=4$.
\end{proof}

\section{Harary graphs}\label{section:harary}

Graph gonality can behave very badly under reasonable graph operations.  For instance, as shown in \cite{sparse} there exist graphs of gonality $3$ with subgraphs of arbitrarily high gonality.  Thus we do not know for certain that $\ci_n(J)\leq \ci_{n+1}(J)$.  So as to obtain such a result at least for even-regular Harary graphs, we prove the following theorem.

\begin{theorem}\label{theorem:gonality_monotone}
Let \(G\) be a vertex-transitive simple graph with even valences, and suppose \(H\) is constructed from \(G\) as follows:  delete a vertex \(v\), and pair up the neighbors of \(v\) in any way, adding an edge between each pair.  Then  \(\gon(H)\leq \gon(G)\).
\end{theorem}    

\begin{proof}
Let \(D\) be a positive rank divisor of degree \(\gon(G)\) on \(G\).  For every vertex \(q\in V(G)\), there exists a collection of legal subset firing moves to place a chip onto \(q\) without debt being present elsewhere.  Choose our \(v\) to be a vertex (not necessarily unique) that requires the {most} subset firing moves to put a chip onto it.  This means that for every \(q\), the firing script to place a chip onto \(q\) never fires \(v\) without firing all of its neighbors; otherwise $v$ must have had a chip to lose, and  we could've gotten a chip to \(v\) in fewer moves than to $q$.  Thus, the neighbors of $q$ always fire at least as many times as \(v\).

Let \(D'\) be the restriction of \(D\) to \(H\).  Note that $\deg(D')=\deg(D)$, since $D(v)$ must have been $0$.  Given \(q\in V(H)\), we can consider the firing script \(c\) that would have eliminated debt on \(q\) in the original graph \(G\); follow the exact same firing script, except of course that there is no \(v\) to fire.

Most vertices will end up with the same number of chips:  they fired the same number of times, as did their neighbors (and their set of neighboring vertices did not change).  Suppose \(w\) was a neighbor of \(v\), which is paired with \(w'\).  Then in \(H\), \(w\) donates the exact same number of chips, since it has the same degree; but it actually receives as many or more chips, since \(w'\) fired at least as many times as \(v\).  Thus following the firing script restricted to \(H\) results in a divisor with at least as many chips on every vertex when compared to \(D\) on \(G\).  That means we end up with an effective divisor with at least one chip on \(q\).  Since \(q\) was an arbitrary vertex of \(H\), we have that \(D'\) has positive rank on $H$, so \(\gon(H)\leq \deg (D')=\deg(D)=\gon(G)\).
\end{proof}

\begin{corollary}
For \(k\) even, we have \(\gon(H_{k,n})\leq \gon(H_{k,n+1})\).
\end{corollary}

\begin{proof}
The graph \(H_{k,n+1}\) is vertex transitive with even valence, and applying the construction in Theorem \ref{theorem:gonality_monotone} to it yields \(H_{k,n}\); in particular, if vertex \(0\) is deleted, connect the vertex \(\ell\) to the vertex \(\ell-\left(\frac{k}{2}+1\right)\).  Applying the theorem yields the desired result.
\end{proof}

Thus to prove that \(\gon(H_{k,n})=2\sum_{i=1}^{k/2} i^2\) for all \(n\) greater than or equal to some \(n_0\), it suffices to prove that it holds for \(H_{k,n_0}\).  In the case of $k=4$, this amounts to showing that $\gon(H_{4,n_0})=10$ for some $n_0$. Using such computational tools as \cite{chip_firing_interface}, we can compute $\gon(H_{4,n})$ for all $n$ up to $n=16$, finding $16$ to be the smallest value of $n_0$ such that $\gon(H_{4,n_0})=10$.  This gives us the following result.

\begin{computeraided}
We have
\begin{eqnarray*}
    \begin{aligned}
        \gon(H_{4,n}) = & 
        \begin{cases}
        \displaystyle \floor*{\frac{n}{4}} + \floor*{\frac{n+1}{4}} + 2,  & n < 16 \\
        10,                                                 & n \geq 16
        \end{cases}
    \end{aligned}
\end{eqnarray*} 
\end{computeraided}

To our knowledge this is the first known instance of an infinite family of graph gonalities being determined through meaningfully computer-aided means. To emphasize the power of this method, we briefly remark that the strongest known general lower bound on graph gonality, the scramble number of a graph \cite{scramble}, is not sufficiently powerful to show that $\gon(H_{4,n})=10$ for any $n$. Indeed, using the screewidth of a graph, one can show that $\sn(H_{4,n})\leq 6$ for all $n$.  We leave the details of this to Appendix \ref{appendix}. 

There is no theory-based reason to stop at $k=4$.  For instance, we know that $\gon(H_{6,n})\leq 28$. To prove that $\gon(H_{6,n})= 28$ for all sufficiently large $n$, we must simply find one integer $n_0$ giving $\gon(H_{6,n_0})=28$.  Since we know $\gon(H_{6,n})\leq n-\alpha(H_{6,n})=n-\left\lfloor \frac{n}{4}\right\rfloor$, we certainly need $28\leq n_0-\left\lfloor \frac{n_0}{4}\right\rfloor$, which necessitates $n_0\geq 37$.  In the best case scenario that $\gon(H_{6,37})=28$, verifying this by brute-force would require checking all effective divisors of degree $27$ on a $37$-vertex graph, checking that none have positive rank.  There are $\binom{27+37-1}{27}=\binom{63}{27}=489,462,003,181,042,451$ such divisors.  This could be pared down for instance by taking advantage of the vertex-transitivity of $H_{6,37}$, e.g. by assuming without loss of generality that $D(v_1)\geq D(v_i)$ for all $i$.  Even with such shortcuts, the number of divisors to check remains massive.

More generally, for $H_{k,n}$, the first moment our bound has a hope of being attained is when
\[\frac{k(k+1)(k+2)}{12}\leq n-\left\lfloor \frac{n}{k/2 + 1}\right\rfloor.\]
Thus $n$ will at the very least need to be on the order of $k^3$ for our upper bound to have a hope of being achieved.

\bibliographystyle{plain}

\appendix

\section{Bounding the scramble number of $H_{4,n}$}\label{appendix}
As discussed in Section \ref{section:harary}, other tools fall short of proving that $H_{4,n}$ has gonality $10$ for sufficiently large $n$.  In particular, the strongest known general lower bound on graph gonality, the scramble number of graphs \cite{scramble}, will not suffice.

We briefly recall the definition of scramble number here.  A \emph{scramble} $\mathcal{S}$ on a graph $G$ is a collection $\mathcal{S}=\{E_1,\ldots,E_s\}$ of nonempty subsets of $V(G)$, such that the subgraph induced by each $E_i$ is connected.  The \emph{hitting number} $h(\mathcal{S})$ of a scramble is the minimum size of a set $S$ of vertices such that $S\cap E_i$ is nonempty for all $i$.  The \emph{egg-cut number} $e(\mathcal{S})$ of a scramble is the minimum size of a set $T$ of vertices such that $G-T$ has two connected components, one containing $E_i$ and the other containing $E_j$ for some $i$ and $j$. (If no such set $T$ exists, we set $e(\mathcal{S})=\infty$.)  The \emph{order} of a scramble is then $||\mathcal{S}||=\min\{h(\mathcal{S}),e(\mathcal{S})\}$.  Finally, the scramble number $\sn(G)$ of $G$ is the maximum order of a scramble on the graph.

It was shown in \cite[Theorem 1.1]{scramble} that $\sn(G)\leq \gon(G)$ for all $G$.  Thus to show that $\gon(H_{4,n})\geq 10$ for some $n$, it would suffice to find a scramble $\mathcal{S}$ of order $10$ on $H_{4,n}$.  We will show that no such $n$ exists.

To upper bound $\sn(H_{4,n})$, we utilize the tool of \emph{screewidth}.  This is defined as follows.  A \emph{tree-cut decomposition} $\mathcal{T}=(T,X)$ of a graph $G$ is a tree $T$ (that is, a connected graph with no cycles) together with a function $X:V(G)\rightarrow V(T)$. To avoid confusion between $G$ and $T$, we refer to the vertices of $T$ as \emph{nodes}, and edges of $T$ as \emph{links}. Pictorially, we represent $\mathcal{T}$ as a thickened copy of the tree $T$, with vertices of $G$ drawn within the associated nodes of $T$, and all edges $e=uv$ of $G$ drawn along the unique path within $T$ connecting the node with $u$ to the node with $v$.  For instance, Figure \ref{figure:tree_cut_decomposition} represents a tree-cut decomposition of the Harary graph $H_{4,14}$.

\begin{figure}[hbt]
    \centering
    \includegraphics{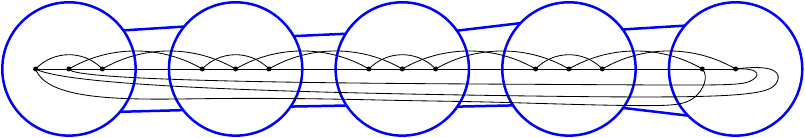}
    \caption{A tree-cut decomposition of $H_{4,14}$}
    \label{figure:tree_cut_decomposition}
\end{figure}
Given a tree-cut decomposition, we compute a number for each link and each node of $T$.  For each link, we count the number of edges of $G$ passing through it. (For our example, this number is $6$ for each link.)  For each node, we count the number of vertices of $G$ within it, and add to this the number of ``tunnelling'' edges from $G$ passing through the node without either endpoint in it. (For our example, this number is $2$ or $3$ for the endpoint nodes, and is $6$ for all others.)  The \emph{width} of $\mathcal{T}$ is the maximum of all these numbers. (So the width of our example is $6$.)  Finally, the \emph{screewidth} of $G$ is the minimum width of any tree-cut decomposition.

It was shown in \cite{screewidth} that $\sn(G)\leq \scw(G)$ for any graph $G$.  This allows us to use screewidth to measure the limitations of scramble number for certain graphs.  In particular, we can use this relationship to prove the following.

\begin{proposition}
    We have $\sn(H_{4,n})\leq 6$.
\end{proposition}
\begin{proof}
    It suffices to give a tree-cut decomposition $\mathcal{T}$ of $H_{4,n}$ of size at most $6$.  Construct a tree $T$ with $\lceil n/3\rceil$ nodes arranged in a path.  Place the vertices $v_1,v_2,v_3$ in the first node, the vertices $v_4,v_5,v_6$ in the second node, and so on, with the final node receiving between $1$ and $3$ vertices.

The only tunneling edges are those connecting $v_{n-1}$ to $v_1$, $v_n$ to $v_1$ and $v_n$ to $v_2$.  Thus the maximum integer associated to any node is $6$ ($3$ for its vertices, $3$ for its tunneling edges).  Similarly, each link has $6$ edges passing through it, $3$ between adjacent vertices in adjacent nodes and the $3$ aforementioned exceptional edges.  Thus $w(\mathcal{T})\leq 6$ (with equality as long as $n\geq 6$).  This completes the proof. 
\end{proof}
\end{document}